\documentclass[12pt
]{amsart}

\usepackage{amsmath}
\usepackage{amsfonts}
\usepackage{amssymb}
\usepackage{hyperref}

\DeclareMathOperator{\Fold}{{Fold}}

\newtheorem{theorem}{Theorem}[section]
\newtheorem{lemma}[theorem]{Lemma}

\theoremstyle{definition}

\newtheorem{algorithm}[theorem]{Algorithm}

\theoremstyle{remark}
\newtheorem{remark}[theorem]{Remark}

\numberwithin{equation}{section}

\begin{document}

\title[
Paley's Theorem for Hankel Matrices
]
{Paley's Theorem\\
for Hankel Matrices\\
via the Schur Test}
\author
[John Fournier]
{John J.F. Fournier}
\address
[John Fournier]
{Department of Mathematics\\
University of British Columbia\\
1984 Mathematics Road\\
Vancouver BC\\
Canada V6T 1Z2}
\email
{fournier@math.ubc.ca}
\author[Bradley Wagner]{Bradley G. Wagner}
\address
[Bradley Wagner]
{Institute for Disease Modeling,
1555 132nd Ave NE, Bellevue WA 98005,
USA}
\email
{bwagner@intven.com}
\thanks{Results mostly announced
at the Peter Borwein Conference in 2008.}
\thanks{The research of the first author
was supported by NSERC grant 4822.}
\thanks{The research of the second author
was supported by an NSERC USRSA award.}


\begin{abstract}
Paley's theorem 
about
lacunary coefficients
of functions in the classical space $H^1$ 
on the unit 
circle
is equivalent to the statement
that certain Hankel matrices define
bounded operators on~$\ell^2$
of the nonnegative integers.
Since
that
statement reduces easily to the case where the entries
in the matrix are all 
nonnegative,
it must be provable
by the Schur test.
We give such
proofs with interesting patterns
in the vectors used in the test, and we recover the best constant in the main case.
We use 
related
ideas to reprove
the characterization of 
Paley
multipliers
from~$H^1$ to~$H^2$.

\end{abstract}

\maketitle

\markleft{
Fournier and Wagner}

\section{Introduction}
\label{sec:intro}

Consider semi-infinite
matrices
with the \emph{Hankel symmetry}
in which
the entries
only depend on the sum of the indices.
We use a method of Schur to prove that matrices like
\[
A_v =
\begin{bmatrix}
v_0
& v_1& 0& v_2& 0& 0& 0& v_3& \cdots\\
v_1& 0& v_2& 0& 0& 0& v_3& 0& \cdots\\
0& v_2& 0&  0& 0& v_3& 0& 0& \cdots\\
v_2& 0&  0& 0& v_3& 0& 0& 0& \cdots\\
0&  0& 0& v_3& 0& 0& 0& 0& \cdots\\
0& 0& v_3& 0& 0& 0& 0& 0& \cdots\\
0& v_3& 0& 0& 0& 0& 0& 0& \cdots\\
v_3& 0& 0& 0& 0& 0& 0& 0& \cdots\\
\vdots& \vdots& \vdots& \vdots& \vdots&
\vdots& \vdots& \vdots&
\ddots&\\
\end{bmatrix},
\]
where there are relatively large gaps
between
most 
nonzero antidiagonals,
act boundedly on~$\ell^2$ when~$v \in \ell^2$.
It remains unclear how to use the Schur test
to prove boundeness, when it holds, for a general Hankel matrix
with nonnegative entries.

To get~$A_v$, fix
a strictly increasing
sequence $(k_j)_{j=0}^\infty$
of 
nonnegative
integers,
and denote its range by $K$.
Let $(v_j)_{j=0}^\infty$ be a sequence
of complex numbers.
Let $a_v$ be the function on the nonnegative integers
with $a_v(k_j) = v_j$ for all~$j$
and $a_v(n) = 0$ when~$n \notin K$.
For any sequence~$(a(n))_{n=0}^\infty$,
let~$H_a$ be the 
matrix with entries
\begin{equation}\label{eq:MakeA}
H_a(m, n) = a(m+n).
\end{equation}
Then let~$A_v = H_{a_v}$.

We start the indices~$m$ and~$n$ at $0$ rather than $1$
to simplify our discussion
of Theorems \ref{th:HankPaley}
and \ref{th:Paley} below.
The sequence~$a_v$ and the matrix $A_v$
also depend on the set $K$,
but we suppress that fact in the notation.
Recall that $K$ is called a {\em Hadamard set} if
\[
k_{j+1} > (1 + \varepsilon)k_j
\]
for some positive
constant~$\varepsilon$ and all values of~$j$.

Most proofs of Theorem~\ref{th:HankPaley} below
proceed via 
its counterpart
Theorem \ref{th:Paley},
rather than working directly with the matrix formulation.
It is easy, however, to convert the proof
in~\cite[
pp.~274--275
]{KP} 
into such a direct proof.
There are references to other elementary proofs
in~\cite{Ben}, but not to one via the Schur test.
We discuss various forms of 
that
test
in Section \ref{sec:Schur},
recall its proof in Section~\ref{sec:proof},
and
use 
it
in Section~\ref{sec:Paley} to prove the 
following statement.

\begin{theorem}\label{th:HankPaley}
Let $K$ be a Hadamard set,
and let $v \in \ell^2$.
Then the matrix $A_v$ represents
a bounded operator on~$\ell^2$.
\end{theorem}

As noted in~\cite{KP}, this
is
equivalent to the 
better-known
result of Paley \cite{Paley} below.
%
Given a function~$f$ in $L^1(-\pi,\pi]$,
denote its Fourier coefficients by
\[
\hat f(n) = \frac{1}{2\pi}
\int_{-\pi}^{\pi} f(t)e^{-int}\, dt.
\]
Denote the set of such functions $f$ for which
$\hat f(n) = 0$ for all $n < 0$ by~$H^1$,
and let $H^2$ be the set of functions in $L^2$
with the same property.

\begin{theorem}\label{th:Paley}
Let $K$ be a Hadamard set.
If $f \in H^1$,
then the restriction of the coefficients
of $f$ to the set $K$ belongs to $\ell^2(K)$.
\end{theorem}

Another way to state this
is that multiplying the Fourier coefficients
of $H^1$ functions by $1_K$, the indicator function
of the set $K$,
gives the coefficients of functions in~$H^2$.
Sequences, like $1_K$, with this multiplier property
are now called {\em Paley multipliers.}
They 
played
a r\^ole in the 
resolution~\cite{PiSzN}
of the similarity problem of Sz.-Nagy and Halmos.
In Section~\ref{sec:multipliers},
we will 
reprove the following fact.

\begin{theorem}\label{th:Paley2}
Let $(a(n))_{n=0}^\infty$
be a sequence with the property that
\begin{equation}\label{eq:supsum2}
\sup_{j\ge0}
\left[
\sum_{2^j \le n < 2^{j+1}} |a(n)|^2
\right] 
< \infty.
\end{equation}
Then $a$ is a Paley multiplier.
\end{theorem}
This 
follows
from another 
statement
that we reprove 
by passing to its Hankel counterpart, and then
using a variant of the Schur test.

\begin{lemma}\label{th:Kothe}
Let $(a(n))_{n=0}^\infty$
be a sequence with the property that
\begin{equation}\label{eq:sumsquaresum}
\sum_{j\ge0}
\left[
\sum_{2^j \le n < 2^{j+1}} |a(n)|
\right]^2 < \infty.
\end{equation}
Then the sequence $(a(n)\hat f(n))_{n=0}^\infty$
belongs to $\ell^1$ for every function $f$ in~$H^1$.
\end{lemma}


Condition \eqref{eq:supsum2}
above is also necessary for $a$ be a Fourier
multiplier from~$H^1$ to $H^2$.
Condition \eqref{eq:sumsquaresum}
is not necessary for the conclusion of Lemma \ref{th:Kothe}.
Instead,
the strictly weaker condition~\eqref{eq:supdouble}
below is 
necessary
if~$a \ge 0$,
and 
sufficient in any case.
That
is 
an unpublished result of Charles Fefferman;
see~\cite[page~264]{AnSh}.

Consider
the corresponding sufficiency result for Hankel matrices.

\begin{theorem}\label{th:fefferman}
Let $(a(n))_{n=0}^\infty$
be a sequence with the property that
\begin{equation}\label{eq:supdouble}
\sup_{M > 0} \left\{\sum_{j=1}^\infty
\left[\sum_{jM \le n < (j+1)M} |a(n)|\right]^2\right\}
< \infty,
\end{equation}
Then the matrix~$H_a$ represents a bounded operator on~$\ell^2$.
\end{theorem}

The proofs 
of this
in \cite{Bonsall, SlSt, SzW} all 
run via the
equivalent statement 
on~$(-\pi, \pi]$ 
and
the duality between~$H^1$ and~$BMO$.
In~\cite[pp.~423--424]{Ben},
Grahame Bennett
asked for an elementary proof. There must be one
using the Schur test, but we have not found it.

In
Sections~\ref{sec:Folding}--\ref{sec:best}, we sharpen Theorem~\ref{th:HankPaley} in the following way.
Denote
the operator norm of~$A_v$ by~$\|A_v\|_\infty$.
By uniform boundedness, 
that
theorem
is equivalent to the existence of a constant~$C_K$
for which
\begin{equation}
\label{eq:PaleyHankelInequality}
\|A_v\|_\infty \le C_K\|v\|_2
\quad\textnormal{for all~$\ell^2$ sequences~$v$.}
\end{equation}
We
use the Schur test and a pattern found in~\cite{FourPaley} to
recover the fact
that
if~$k_{j+1} > 2k_j$ for all~$j$, then the smallest such constant~$C_K$ is~$\sqrt 2$.

\section{Schur tests}\label{sec:Schur}

Let $F$ be the set of all finitely-supported sequences
in the unit ball of~$\ell^2$ of the nonnegative integers.
Given a semi-infinite matrix $A$, consider
the extended real number
\begin{equation}\label{eq:bilinear}
\|A\|_\infty = \sup\left\{
\left|\sum_{m=0}^\infty \sum_{n=0}^\infty
A(m,n)g(m)h(n)\right|: g, h \in F\right\}
\end{equation}
This gives the usual operator norm
when $A$ represents a bounded operator
on $\ell^2$, and it is equal to $\infty$ otherwise.

We study conditions on the size
of the entries in a Hankel matrix that make its operator
norm finite. In this analysis,
we may assume that the entries in the matrix
are nonnegative, since the supremum above
does not decrease if we 
replace
all entries
in the matrix $A$ and the sequences $g$ and $h$
by their absolute values.

\begin{lemma}\label{th:Schur}
The following conditions on a semi-infinite matrix $A$
and its 
adjoint~$A^*$ are equivalent
when their entries are all nonnegative.
\begin{enumerate}
\item\label{en:boundedness}
$A$ represents a bounded operator
on $\ell^2$ with norm at most $S$.
\item\label{en:factor}
For each number $T > S$ there are matrices $B$ and $C$
with nonnegative entries with the following properties.
\begin{enumerate}
\item\label{en:entries}
$A(m,n) = \sqrt{B(m,n)C(m,n)}$
for all $m$ and $n$.
\item\label{en:rows}
Each row of $B$ has $\ell^1$ norm at most $T$.
\item\label{en:columns}
Each column of $C$ has $\ell^1$ norm at most $T$.
\end{enumerate}
\item\label{en:Schur}
For each number $T > S$ there are
strictly 
positive vectors $u$ and~$w$
for which $Au \le Tw$ and $A^*w \le Tu$.
\end{enumerate}
\end{lemma}

Condition~\eqref{en:Schur}
is the frequently-used {\em Schur test.}
The less-known condition~\eqref{en:factor}
is easier to apply to  Paley multipliers.
We rewrite it using the mixed-norm notation
\[
\|B\|_{(1,\infty)}
= \sup_m \sum_{n=0}^\infty |B(m,n)|
\]
and the notation $B\star C$
for the Schur (or Hadamard) product matrix
with entries $B(m,n)C(m,n)$.
Then condition~\eqref{en:factor} requires that
\begin{equation}\label{eq:factors2}
A\star A = B\star C,
\quad
\|B\|_{(1,\infty)} \le T,
\quad
\|C^*\|_{(1,\infty)} \le T,
\end{equation}
and that 
the entries
in these matrices 
all
be
nonnegative.

The fact that condition \eqref{en:Schur}
implies condition \eqref{en:factor}
is implicit in 
proofs,
like the one in~\cite[Section 5.2]{Nik},
of the Schur test. So is
the fact 
that~\eqref{en:factor} implies \eqref{en:boundedness}.
The fact that
\eqref{en:boundedness}
implies
\eqref{en:Schur}
goes back to~\cite{Kar} or~\cite{Gag}.
Since we use ideas from 
a
proof of the lemma,
we
include 
that
proof.

\begin{remark}
For some authors~
\cite[pp. cii--cviii]{DK}, the~``Schur test''
requires
more, namely
that
conditions~\eqref{en:rows}
and~\eqref{en:columns} hold with~$B$ and~$C$ replaced by~$A$.
The former condition then implies that~$A$ acts boundedly on~$\ell^\infty$,
and the latter that~$A$ acts boundedly on~$\ell^1$.
One can then use Cauchy-Schwartz, or more sophisticated
methods, to get boundedness on~$\ell^2$.
\end{remark}

\begin{remark}
Sometimes, one can exhibit strictly positive vectors~$u$ and~$v$ for which condition~\eqref{en:Schur} holds
in the stronger form where~$Au = Sw$ and~$A^*w = Sv$. Then~$u$ is an eigenvector of~$A^*A$ with eigenvalue~$S^2$, so 
that~$\|A\|_\infty \ge S$.
Since 
the Schur test makes~$S$ an upper bound
for that norm,~$\|A\|_\infty = S$ in such a case.
Moreover,
if~$A$ is symmetric then~$u + w$ is an eigenvector of~$A$ with eigenvalue~$S$.
We exploit this possibility in Section~\ref{sec:best}. 
\end{remark}

\section{
A proof of equivalence}
\label{sec:proof}
 
When condition~\eqref{en:Schur}) in Lemma~\ref{th:Schur} holds, let
\begin{equation}\label{eq:RankOne}
B = A\star\left(\frac{u(n)}{w(m)}\right)_{m, n = 0}^\infty,
\quad\text{and}\quad
C = A\star\left(\frac{w(m)}{u(n)}\right)_{m, n = 0}^\infty
\end{equation}
respectively. Then~$B\star C = A\star A$.
For this choice of~$B$ and~$C$, 
conditions~\eqref{en:rows} and~\eqref{en:columns}
are 
equivalent to requiring that~$Au \le Tw$
and $A^*w \le Tu$. 

Suppose 
now
that condition \eqref{en:factor} holds in any way.
Apply the Cauchy-Schwarz inequality to the sum
\[
\sum_{m=0}^\infty \sum_{n=0}^\infty
A(m,n)g(m)h(n)
= \sum_{m=0}^\infty \sum_{n=0}^\infty
\sqrt{B(m, n)}g(m)\sqrt{C(m, n)}h(n)
\]
to get the upper bound
\[
\left[\sum_{m=0}^\infty \sum_{n=0}^\infty
{B(m, n)}g(m)^2\right]^{1/2}
\left[\sum_{m=0}^\infty \sum_{n=0}^\infty
{C(m, n)}h(n)^2\right]^{1/2}.
\]
By condition \eqref{en:rows}, the inner
sum in the first factor above
is no larger than~$\{[Tg](m)\}^2$.
Since$\|g\|_2 \le 1$, the first factor
is at most $\sqrt{T}$.
Reversing the order of summation in the second factor
and using condition~\eqref{en:columns} gives
the same upper bound for that factor
when $\|h\|_2 \le 1$.
The matrix~$A$ therefore represents
a bounded operator on $\ell^2$ with norm at most $T$.
Use this for all numbers $T > S$
to get that $\|A\|_\infty \le S$, and that $\eqref{en:factor}
\implies \eqref{en:boundedness}$.

Finally, assume that condition \eqref{en:boundedness}
holds, and fix $T > S$. Suppose initially that $A$
is symmetric. Given any strictly
positive sequence~$d$ in~$\ell^2$, let $u$
be the sum of the convergent series
\begin{equation}\label{eq:geometric}
d + \frac{Ad}{T}+ \frac{A^2d}{T^2}
+ \dots
\end{equation}
Then $Au \le Tu$, and condition \eqref{en:Schur} holds
with $w = u$.

When $A$ is not symmetric,
apply the reasoning above to~$A^*A$
and to~$AA^*$, both with norms $S^2$,
to get strictly positive vectors $\hat u$
and $\check u$ 
for which
\[
A^*A\hat u \le T^2\hat u,
\text{\ and\ }
AA^*\check u \le T^2\check u.
\]
Then let
\begin{equation}
u = T\hat u + {A^*\check u},
\text{\ and\ }
w = {A\hat u} + T\check u,
\end{equation}
and check
that
$Au \le Tw$
and $A^*w \le Tu$.
So $\eqref{en:boundedness} \implies
\eqref{en:Schur}$.
\qed

\begin{remark}
\label{rm:dichotomies} 
 
There are several dichotomies here.
First,
the proof above shows that 
if condition~\eqref{en:factor}
holds,
then it must hold with
factors~$B$ and~$C$ equal to Schur products of~$A$
with matrices of rank one 
in which
all entries are positive.

Next, when~$A$ is symmetric and acts boundedely on~$\ell^2$,
the proof yields conditions~\eqref{en:Schur}
and~\eqref{en:factor} with~$u=w$ and~$C = B^*$.
But if condition~\eqref{en:Schur}
or~\eqref{en:factor} is satisfied in any way
for a symmetric matrix with nonnegative entries,
then boundedness follows, and both conditions
can be satisfied symmetrically.
This also follows by simply replacing the strictly positive vectors~$u$ and~$w$ that satisfy condition~\eqref{en:Schur} with~$u+w$.

Finally, the proof that $\eqref{en:boundedness}
\implies \eqref{en:Schur}$ yields strictly positive vectors
in~$\ell^2$ that satisfy condition \eqref{en:Schur},
while the proof that $\eqref{en:Schur}
\implies \eqref{en:boundedness}$
only requires that~$u$ and  $w$ be strictly positive,
without necessarily belonging to $\ell^2$.
This occurs in~\cite{KO} and~\cite{Bor},
where the sequence~$(1/\sqrt{n+1})$
is used
to prove one of the
Hilbert inequalities
with a best constant.
Some of the sequences that we use in Section~\ref{sec:best}
also do not belong to~$\ell^2$

\end{remark}

\section{Proving Paley}
\label{sec:Paley}

Since sums of bounded operators are bounded, the conclusion of Theorem~\ref{th:HankPaley} also holds for sets~$K$ that are unions of finitely many Hadamard sets.
Similarly, it suffices to prove
the theorem
for sets~$K$ that are \emph{strongly lacunary} in the sense that
\begin{equation}
\label{eq:strong}
k_{j+1} > 2k_j
\quad\text{for all }j,
\end{equation}
because each Hadamard set is a union of finitely many strongly lacunary sets.

We 
offer
three
answers to the question
``What 
do
Schur-test 
proofs
of Paley's theorem
look like?''
In this section,
we
specify a simple factorization
that satisfies condition~\eqref{en:factor}
in Lemma~\ref{th:Schur}
whenever $K$ is a union
of finitely-many Hadamard sets.
In Section~\ref{sec:Folding},
we
use 
that factorization 
when~$K$ is strongly lacunary
to get
a 
strictly positive 
vector $u$
that satisfies the symmetric version
\begin{equation}
\label{eq:SymmetricSchur}
A_v u \le Tu
\end{equation}
of condition~\eqref{en:Schur}
in the lemma.
We discuss the pattern in this example in Section~\ref{sec:Patterns}.
In Section~\ref{sec:best},
we 
modify 
the
example to recover the best constant
in the inequality~$\|A_v\|_\infty \le C_K\|v\|_2$
when~$K$ is strongly lacunary.

A set is a union of finitely many Hadamard sets
if and only
there is a constant $M$
so that there are at most $M$ 
members of that set
in each 
interval $[m, 2m)$.
As in~\cite{RudPal}, this
property also characterizes
sets~$K$
for which the conclusion of Theorem~\ref{th:Paley} holds.
Similar reasoning applies to Theorem~\ref{th:HankPaley}.

For any such set~$K$,
get~$B$ and~$C$ from~$A_v$ as follows.
Let~$B$ and~$C$ both match~$A_v$ on the main diagonal.
Let~$B$ match~$A_v\star A_v$ above that diagonal.
In that region, let~$C(m,n) = 1$ on the special antidiagonals where~$m+n = k_j$ for some~$j$,
and let~$C(m,n) =0$ otherwise.
Reverse the r\^ oles of~$B$ and~$C$ below the main diagonal.
Then $B\star C = A_v\star A_v$, and~$C = B^*$.

Since
$
\|C^*\|_{1,\infty} = \|B\|_{1,\infty},
$
it suffices to show that either mixed norm is finite.
In each row sum in~$B$,
the diagonal term is at most $\|v\|_2$, while the sum
of the terms
to the right of it is at most~$(\|v\|_2)^2$.
The sum 
of the terms
to the left of the diagonal
is the number of~$1$'s below the diagonal
in that row of the matrix $B$. In the $m$-th row,
this is the number of indices $k_j$
in the interval $[m, 2m)$.
Therefore
\begin{equation}
\label{eq:MixedUpperBound}
 \sum_{n=0}^\infty B(m,n)
 \le M + \|v\|_2 + (\|v\|_2)^2. \qquad\qed
\end{equation}




\section{Paley multipliers}
\label{sec:multipliers}

\begin{proof}[Proof of Lemma~\ref{th:Kothe}]

The same arguments giving the equivalence of Theorems~\ref{th:Paley}
and~\ref{th:HankPaley} show that Lemma~\ref{th:Kothe}
is equivalent to the statement that if a sequence~$a$ satisfies
condition~\eqref{eq:sumsquaresum}, then the Hankel matrix~$H_a$
acts boundedly on~$\ell^2$.
It suffices to prove that boundedness for strictly positive sequences~$a$.

In that case, apply
condition~\eqref{en:factor} 
in Lemma~\ref{th:Schur}
as follows.
Let~$B$ match~$H_a$ along the main diagonal.
Above that diagonal, get the entries in~$B$ by multipling the entries
in each column of~$H_a$ by the sum  along that column above the diagonal.
In the~$n$-th column, for instance,
multiply above the diagonal by
\[
\sum_{0 \le m < n} H_a(m,n)
= \sum_{n \le r < 2n} a(r).
\]
To the left of the diagonal,
divide the entries in the~$n$-th row of~$H_a$ by the same 
positive
sum.
Let~$C = B^*$, and check that~$H_a\star H_a = B\star C$.

Split
the~$\ell^1$ norm of the~$n$-th row of~$B$ 
into
the part before the diagonal, the diagonal term,
and the part after the diagonal. The division above
makes the the 
norm of the
first part equal to~$1$. The diagonal term
is no larger than~$\|a\|_2$. The sum beyond the diagonal
is equal to
\[
\sum_{s > n} H_a(n, s)
\left[\sum_{s \le r < 2s} a(r)\right]
= \sum_{s > n} a(n+s)
\left[\sum_{s \le r < 2s} a(r)\right].
\]
In the inner sum above, the index~$r$ is always greater than~$(n+s)/2$
and smaller than~$2(n+s)$. So the double sum is majorized by
\[
\sum_{i=1}^\infty a(i)
\left[\sum_{i/2 < r < 2i} a(r)\right],
\]
which does not depend on~$n$. To
bound
this expression,
write it dyadically as
\begin{equation}
\label{eq:expression}
\sum_{j\ge0}
\left\{
\sum_{2^j \le i < 2^{j+1}} a(i)
\left[\sum_{i/2 < r < 2i} a(r)\right].
\right\}
\end{equation}
Here the index~$r$ in the innermost sum
is always greater than~$2^{j-1}$
and smaller than~$2^{j+2}$. So the 
quantity~\eqref{eq:expression}
is majorized by
\[
\sum_{j\ge0}
\left\{
\sum_{2^{j-1} < i < 2^{j+2}} a(i)
\right\}^2,
\]
which is finite if condition~\eqref{eq:sumsquaresum} holds.
That gives the desired uniform bound on the~$\ell^1$ norms of the rows
of~$B$.
\end{proof}

\begin{proof}[Proof of Theorem~\ref{th:Paley2}]
Multiplying coefficients of~$H^1$ functions by a sequence~$b$
always gives an~$\ell^2$ sequence
if and only if the product sequences~$bc$
with~$\ell^2$ sequences~$c$ are always multipliers
from~$\widehat{H^1}$ to~$\ell^1$. By Lemma~\ref{th:Kothe},
the product sequences have this property if
\[
\sum_{j\ge0}
\left[
\sum_{2^j \le n < 2^{j+1}} |b(n)||c(n)|
\right]^2 < \infty.
\]
If condition~\eqref{eq:supsum2}
holds with the sequence~$a$ replaced by~$b$,
then the Cauchy-Schwartz inequality yields
the condition above for every~$\ell^2$ sequence~$c$.
\end{proof}

\section{
Folding patterns}
\label{sec:Folding}

When~$K$ is strongly lacunary,
and~$v$ is strictly positive,
ask 
that
the matrix $B$
in Section~\ref{sec:Paley}
satisfy the 
version of
equation~\eqref{eq:RankOne} with~$w = u$,
that is
\begin{equation}
\label{eq:symmetric}
B = A_v\star\left(\frac{u(n)}{u(m)}\right)_{m,n = 0}^\infty.
\end{equation}
Unless~$m+n = k_j$ for some~$j$, the entries~$B(m, n)$ and~$A_v(m, n)$ both vanish, and
condition~\eqref{eq:symmetric} then places no restriction on the nonzero numbers~$u(n)$ and~$u(m)$.
There is also no restriction
if~$m = n = k_j/2$,
when that is an integer,
since~$B$ and~$A_v$ agree along the main diagonal.

In the remaining cases,~$A(m, n) = v_j$,
while~$B(m, n) = 1$ if~$(m, n)$ lies below the main diagonal, and~$B(m, n) = v_j^2$ above that diagonal.
In the former
case, equation~\eqref{eq:symmetric} yields
that~$u(m) = v_iu(n)$,
that
is
\begin{equation}
\label{eq:solved}
 u(k_j-n) = v_ju(n)
 \quad\text{when~$0 \le n < k_j/2.$}
 \end{equation}
Applying condition~\eqref{eq:symmetric} above the main diagonal gives the equivalent conclusion
with~$m$ and~$n$ interchanged.

Rescale the unknown positive vector $u$ to make $u(0) = 1$.
Formula~\eqref{eq:solved} remains valid.
Apply it
when~$n=0$
to get that
\[
u(k_j) = v_ju(0) = v_j
\quad\text{for all}\quad j.
\]
By 
strong lacunarity,~$k_i < k_j/2$
when~$i < j$. 
By formula~\eqref{eq:solved},
\[
u(k_j-k_i) = v_ju(k_i) = v_jv_i
\quad\text{for all}\quad i < j.
\]
The fact that~$k_i - k_h < k_j/2$ when~$h < i < j$
then gives that
\[
u(k_j - k_i + k_h) = v_jv_iv_h
\quad\text{when}\quad
h < i < j.
\]
Continue in this way.

Strong lacunarity
implies that each nonnegative integer,~$k$ say,
has at most one 
representation
as 
a,
possibly empty,
sum
of the form
\begin{equation}\label{eq:alternating}
k = k_{j_1} - k_{j_2} + k_{j_3} - \cdots \pm k_{j_r}
\quad\text{where}\quad
j_1 > j_2 > j_3 > \cdots.
\end{equation}
When~$k$ has that form, the analysis above shows that
\begin{equation}\label{eq:product}
u(k) = v_{j_1}v_{j_2}v_{j_3}\cdots v_{j_r}.
\end{equation}
Denote the set of integers with
an alternating representation~\eqref{eq:alternating}
by~$\Fold(K)$.
If~$k_j = 2^j - 1$ for all~$j$,
then that set consists of all nonnegative integers.
In that case,
formula~\eqref{eq:product},
with the usual convention that the empty product is~$1$,
completely determines the sequence~$u$.
Moreover,~$u$
is 
strictly positive since~$v$ is. It suffices to prove
the theorem for such~$v$'s.

When~$k_0 > 0$, or~$k_{j+1} > 2k_j + 1$
for some value(s) of~$j$, proceed as follows.
If~$0 \notin K$, augment~$K$ with~$0$, and then reindex~$K$.
For each nonnegative integer~$j$,
let~$L_j$ and~$R_j$ be the sets of integers
in the intervals~$[0, k_j]$ and~$[k_{j+1} - k_j, k_{j+1}]$
respectively. 
Strong lacunarity
makes~$L_j$ and~$R_j$ disjoint. They cover~$L_{j+1}$
if and only if~$k_{j+1} = 2k_j + 1$.

\begin{algorithm}\label{th:fold}
Start with~$u(0) = 1$,
and define~$u$ iteratively by imposing the following conditions
in the set~$L_{j+1}$.
\begin{enumerate}
\item{}\label{en:coefficients}
The values of~$u$ on~$R_j$
are just those on~$L_j$ listed
in reverse order and multiplied by~$v_{j+1}$.
\item{}\label{en:positive}
At integers, if any, in the 
first half~$(k_j, k_{j+1}/2]$ of the
gap between~$L_j$ and~$R_j$
let~$u$ take any 
positive 
values.
\item{}\label{en:SecondHalf}
In the part of that gap strictly the right of the midpoint~$k_{j+1}/2$, use the 
values from the part of the gap strictly the left of~$k_{j+1}/2$ listed in reverse order and multiplied by~$v_{j+1}$.
\end{enumerate}
\end{algorithm}
Then~$u$
satisfies condition~\eqref{eq:symmetric},
and
equation~\eqref{eq:product} holds
on the set~$\Fold(K)$.

\begin{remark}
\label{rm:simpler}
A simpler way to satisfy condition~\eqref{en:Schur}
in Lemma~\ref{th:Schur}
is to
allow~$u$ to take any fixed positive constant value
in each gap,
rather than satisfying condition~\eqref{en:SecondHalf} above.
The corresponding matrix~$B$ then differs from the one used in Section~\ref{sec:Paley}.
\end{remark}

\section{Related Patterns}
\label{sec:Patterns}

Patterns with folding properties similar
to those in Algorithm~\ref{th:fold}
have occurred in settings
where the sequence~$u$ does not have
to be nonzero or even real-valued. In those cases, let~$u$ vanish
in any gaps between the sets~$L_j$ and~$R_j$,
instead of being positive there as in 
the algorithm.
Formula~\eqref{eq:product} still gives the values of~$u$
on~$\Fold(K)$,
and~$u$ vanishes 
off that set.


For any sequence~$(v_j)_{j \ge 0}$
of complex numbers,
define
functions~$U_e^{(j)}$
and~$U_o^{(j)}$ on the interval~$(-\pi, \pi]$ as follows.
\begin{algorithm}\label{th:refold}
Let~$U_e^{(0)} = 1$ and~$U_o^{(0)} = 0$.
Given~$U_e^{(j-1)}$ and~$U_o^{(j-1)}$, let
\begin{align}
\label{eq:evenU's}
U_e^{(j)}(t) &= U_e^{(j-1)}(t) + {
\overline{v_{j}\exp({ik_{j}t})}}
{U_o^{(j-1)}}(t),\\
\textnormal{and}\qquad
U_o^{(j)}(t) &= U_o^{(j-1)}(t) + {v_{j}\exp({ik_{j}t})}
{U_e^{(j-1)}}(t).
\end{align}
\end{algorithm}

These 
trigonometric
polynomials connect with Theorems~\ref{th:HankPaley} and~\ref{th:Paley}
in two ways. 
First, 
rational functions
of the form~$U_o^{(j)}/U_e^{(j)}$,
with~$v$ replaced by a related sequence
obtained via the Schur \emph{algorithm} rather than the Schur test,
were used in~\cite{FourPaley} to prove Theorem~\ref{th:Paley}
and 
to
discover
a
useful
extension of it.
Second,
for a nonnegative sequence~$v$, the coefficients of the function 
\[
t \mapsto U_e^{(j)}(-t) + U_o^{(j)}(t)
\]
match the sequence~$u$
on~$\Fold(K)\cap L_j$,
and they vanish otherwise.
As
noted in~\cite{FourPaley}, a very similar folding pattern for coefficients occurs in the modification of the Rudin-Shapiro polynomials used in~\cite{ClunieRS} and~\cite{FournierRS};
the only difference 
is that
the plus sign in equation~\eqref{eq:evenU's} is replaced
with a minus sign.



\begin{remark}
If~$k_j = 2^j - 1$ for all~$j$, then~$u$
can be represented by a product
that takes the place of the series~\eqref{eq:geometric}
that was
used in the proof of Lemma~\ref{th:Schur}.
Split the matrix~$A_v$ as a formal sum
\begin{equation}
\label{eq:split}
A_v = \sum_{j=1}^\infty A_v^{(j)},
\end{equation}
where~$A_v^{(j)}$ matches~$A_v$ on the antidiagonal where~$m+n = k_j$,
and vanishes elsewhere. Let~$u^{(J)}$ be
the sequence obtained by stopping
Algorithm~\ref{th:fold} when~$j=J$.
Let~$e^{(0)}$
be the transpose of~$(1, 0, 0, 0, \cdots)$.
Then
\begin{equation}\label{eq:partialproduct}
u^{(J)} = \left[I + A_v^{(J)}\right]
\left[I + A_v^{(J-1)}\right]
\left[\cdots\right]
\left[I + A_v^{(1)}\right]
e^{(0)}.
\end{equation}
Expand this
as 

\begin{equation}\label{eq:partproducts}
\left(
I + \sum_{j=1}^J\left\{
A_v^{(j)}
\left[I + A_v^{(j-1)}\right]
\left[\cdots\right]
\left[I + A_v^{(1)}\right]\right\}\right)
e^{(0)}.
\end{equation}
The~$j$-th matrix summand above times~$e^{(0)}$
vanishes off the set~$R_j$.
Since these sets are disjoint,~$u$ is equal to the infinite product
\[
\left[\cdots\right]\left[I + A_v^{(J)}\right]
\left[I + A_v^{(J-1)}\right]
\left[\cdots\right]
\left[I + A_v^{(1)}\right]
e^{(0)}.
\]

\end{remark}

\section{Recovering the best constant}
\label{sec:best}

Continue to assume
that~$K$ is strictly lacunary
and contains~$0$.
In proving
that
\begin{equation}
\label{eq:specialbest}
\|A_v\|_\infty \le \sqrt2\|v\|_2
\end{equation}
in that case,
it is enough to consider nontrivial 
nonnegative
sequences~$v$ with finite support, rescale them so that~$\|v\|_2 = 1/\sqrt2$,
and show that~$\|A_v\|_\infty \le 1$.

We
will
apply Algorithm~\ref{th:fold}
to a sequence~$c$ obtained from such a sequence~$v$,
and use the Schur test to confirm that~$\|A_v\|_\infty \le 1$.
We will use the
inverse process going from~$c$ to~$v$
to exhibit
cases where~$\|A_v\|_\infty = 1$ while~$v$ exceeds~$1/\sqrt2$ by as little as we like.
So
the constant~$\sqrt2$
in inequality~\eqref{eq:specialbest}
is best possible.

The inverse process is easier to describe.
As in~\cite{FourPaley}, given~$(c_j)_{j=0}^\infty$,
form sequences~$v^{(J)}(c)$ as follows.
Let~$v^{(0)}(c)$ be the transpose of~$(c_0, 0, 0, \cdots)$.
Given~$v^{(J-1)}(c)$ pass to~$v^{(J)}(c)$ by 
setting the~$J$-th component of~$v^{(J)}(c)$
equal to~$c_{J}$, and multiplying
all earlier entries in~$v^{(J-1)}(c)$ by~$(1 - |c_{J}|^2)$.
%

Inequality~\eqref{eq:specialbest} and the fact that~$\sqrt2$ is the best constant there follow easily from the two lemmas below.
The first 
one
is equivalent, in the usual way, to  previous
results~\cite{FourPaley}
about functions on~$(-\pi, \pi]$.
Here, we prove it more directly using the Schur test.
For completeness, we also include a 
modified
proof of the second lemma, which is essentially in~\cite{FourPaley}.
\begin{lemma}
\label{th:HankelSchur}
If~$|c_j| \le 1$ for all~$j \le J$, 
then~$\left\|A_{v^{(J)}(c)}\right\|_\infty \le 1$.
If~$c_j$ is real for all~$j \le J$, and~$c_0 = 1$,
then~$\left\|A_{v^{(J)}(c)}\right\|_\infty = 1$.
\end{lemma}

\begin{lemma}
\label{th:Range}
If~$\|v\|_2 \le 1/\sqrt2$ and~$v_j = 0$ for all~$j > J$,
then~$v = v^{(J)}(c)$ for a sequence~$c$ with the property that~$0 \le c_j \le 1/\sqrt2$ for all~$j$.
On the other hand,~$\|v^{(J)}(c)\|_2 =
\sqrt{(J+2)/(2J+2)}$
if~$c_j = 1/\sqrt{j+1}$ for all~$j$.
\end{lemma}

\begin{proof}[Proof of Lemma~\ref{th:Range}]
Both parts are clear when~$J=0$.
Let~$J > 0$, and assume that both hold when~$J$ is replaced by~$J-1$.
Given a sequence~$v$
for which~$|v_j| < 1$ for all~$j > 0$
and~$v_j = 0$ for all~$j > J$,
replace~$v_J$ by~$0$ and divide
all earlier entries in~$v$ by~$(1 - |v_J|^2)$
to get a sequence~$v'$.

Let~$\varepsilon = \left\|v\right\|^2 - 1/2$
and~$\varepsilon' = \left\|v'\right\|^2 - 1/2.$
Then
\[
\varepsilon =
\left\{\left(1 - |v_J|^2\right)^2\left[\frac{1}{2} + \varepsilon'\right]
+ |v_J|^2\right\} - \frac{1}{2}.
\]
Expand the product~$\left(1 - |v_J||^2\right)^2(1/2)$ and simplify to get that
\begin{equation}
\label{eq:NextEpsilon}
\varepsilon = \frac{1}{2}|v_J|^4
+ \left(1 - |v_J|^2\right)^2\varepsilon'.
\end{equation}

It follows that if~$\varepsilon' >0$,
then~$\varepsilon > 0$, contrary to the assumption 
in the first part of the lemma
that~$\|v\|_2 \le 1/\sqrt2$.
Therefore,~$\varepsilon' \le 0$, and~$\|v'\|_2 \le 1/\sqrt2$ 
in that part.
By the inductive assumption~$v'$ is equal to
$v^{(J-1)}(c)$
for a sequence with the property that~$|c_j| \le 1/\sqrt2$ for all~$j$.
Replacing~$c_J$ by~$v_J$ then makes~$v = v^{(J)}(c)$.

In the second part of the lemma, apply
formula~\eqref{eq:NextEpsilon}
when~$v = v^{(J)}(c)$ and~$v' = v^{(J-1)}(c)$. By the inductive assumption,~$\varepsilon'
= 1/(2J)$. Hence
\begin{gather*}
 \varepsilon = \frac{1}{2(J+1)^2}
 + \left[1 - \frac{1}{J+1}\right]^2\frac{1}{(2J)}\\
 = \frac{1}{2(J+1)^2}
 + \left[\frac{J}{J+1}\right]^2\frac{1}{(2J)}
 = \frac{1}{2(J+1)}.
\qedhere
\end{gather*}
\end{proof}

\begin{proof}[Proof of Lemma~\ref{th:HankelSchur}]
In the first part of the lemma, replacing~$c$ by~$|c|$ replaces~$v^{(J)}(c)$ by~$\left|v^{(J)}(c)\right|$, and does not decrease~$\left\|A_{v^{(J)}(c)}\right\|_\infty$.
Small enough changes in~$(c_0, \cdots c_J)$
change~$\left\|A_{v^{(J)}(c)}\right\|_\infty$
by as little as we like, and
that norm
is not affected by changes in~$c_j$
when~$j > J$. So we may assume
that~$0 < c_j < 1$ for all~$j$.

Build a strictly positive sequence~$u$ by using Algorithm~\ref{th:fold}
with the sequence~$v$ replaced by~$c$. Since the matrix entries~$A_{v^{(J)}(c)}(m, n)$ vanish when~$m + n > k_J$, each product vector~$A_{v^{(J)}(c)}u$ is finite;
moreover~$\left[A_{v^{(J)}(c)}u\right](m) = 0$ 
for all~$m > k_J$.

The desired upper bound on~$\left\|A_{v^{(J)}(c)}\right\|_\infty$ follows
by the Schur test
if
\begin{equation}
 \label{eq:reduced}
 A_{v^{(J)}(c)}u \le u.
\end{equation}
Since~$|v_0| \le 1$,
inequality~\eqref{eq:reduced}
is clear
when~$J=0$. Assume that it holds when some positive value~$J$ is replaced by~$J-1$.

Let~$P_J$ be the Hankel matrix with entries equal to~$1$
on the antidiagonal where~$m + n = k_J$ and to~$0$
otherwise. Then
\[
 A_{v^{(J)}(c)} = (1 - c_J^2)A_{v^{(J-1)}(c)} + c_JP_J.
\]
Since the matrix entries~$A_{v^{(J-1)}(c)}(m, n)$ vanish when~$m + n > k_{J-1}$,
and~$A_{v^{(J-1)}(c)}u \le u$,
matters reduce to showing that
\begin{equation}
\label{eq:TwoCases}
 \left[c_JP_Ju\right](m) \le
 \begin{cases}
  c_J^2u(m)	&\text{when $m \le k_{J-1}$;}\\
  u(m)		&\text{otherwise.}
 \end{cases}
\end{equation}

It will turn out that equality holds in the first case above and also when~$k_J/2 < m \le k_J$,
while strict inequality holds otherwise.
Multiplying
the vector~$u$ by the matrix~$P_J$
lists the 
entries~$(u_0, u_1, \cdots u_{J})$ in reverse order,
and annihilates all other entries.
So the inequality for the second case above is
strict
when~$m > k_J$,
because~$\left[c_JP_Ju\right](m)$ vanishes
in that subcase, but~$u(m) > 0$.

To confirm equality in the subcase where~$k_J - k_{J-1} \le m \le k_J$, that is 
when~$m \in R_{J-1}$, recall
that~$u(m)$ is defined 
there
by listing the values of~$u$ on~$L_{J-1}$
in reverse order
and multiplying them by~$c_J$.
Forming~$\left[c_JP_Ju\right](m)$ does the same things.

The definition~$u(m)$ also makes~$\left[c_JP_Ju\right](m) = u(m)$
inside the second half of the gap
where~$k_{J-1} < m < k_J - k_{J-1}$.
In the first half of that gap, where~$k_{J-1} < m \le k_J/2$, the values~$\left[c_JP_Ju\right](m)$
are equal to~$c_Ju(k_j - m).$
This is equal to~$c_Ju(m)$ if~$m = k_J/2$,
and to~$c_J^2u(m)$ if~$m < k_J/2$,
making~$\left[c_JP_Ju\right](m) < u(m)$ in either case.
The analysis in the gap~$(k_{J-1}, k_J)$ is  even simpler
for the choice of~$u$ proposed in Remark~\ref{rm:simpler},
when there is strict inequality in the whole gap.

Finally, suppose that~$m \le k_{J-1}$, that is~$m \in L_{J-1}$. The values of~$c_JP_Ju$ on~$L_{J-1}$ are those of~$u$ on~$R_{J-1}$ listed in reverse order and multiplied by~$c_J$. On the other hand, the values of~$u$ on~$R_{J-1}$ come 
from those on~$L_{J-1}$ by another reversal of order and another multiplication by~$c_J$.
So~$\left[c_JP_Ju\right](m) = c_J^2u(m)$ again,
and inequality~\eqref{eq:TwoCases} holds.

The conclusion
that~$\left\|A_{v^{(J)}(c)}\right\|_\infty = 1$
in the second part of the lemma now follows if~$1$
is an eigenvalue~$A_{v^{(J)}(c)}$. For that purpose,
let~$u^{(J)}$ be the vector obtained from~$u$ by replacing all values of~$u$ off~$\Fold(K)\cap L_J$ by~$0$.
Now~$A_{v^{(J)}(c)}u^{(J)} = u^{(J)}$ if~$J=0$,
because~$v_0 = 1$.

Suppose that~$A_{v^{(J-1)}(c)}u^{(J-1)} = u^{(J-1)}$.
Then
\[
 [A_{v^{(J-1)}(c)}u^{(J)}](m) =
 \begin{cases}
  u^{(J)}(m) &\text{when~$m \le k_{J-1}$}\\
  0 & \text{otherwise.} 
 \end{cases}
\]
Proving that~$A_{v^{(J)}(c)}u^{(J)} = u^{(J)}$
therefore reduces to checking that
\[
  \left[c_JP_Ju^{(J)}\right](m) =
 \begin{cases}
  c_J^2u^{(J)}(m)	&\text{when $m \le k_{J-1}$;}\\
  u^{(J)}(m)		&\text{otherwise.}
 \end{cases}
\]
Both sides of the equation above vanish when~$m > k_J$
and also in the gap where~$k_{J-1} < m < k_J - k_{J-1}$.
For the same reasons as before, the two sides agree when~$k_J - k_{J-1} \le m \le k_J$
and when~$m \le k_{J-1}$.
\end{proof}

\begin{remark}
\label{rm:Include0}
Inequality~\eqref{eq:specialbest} also holds for strongly lacunary sets that do not contain~$0$, because it does when~$0 \in K$ and~$v_0 = 0$. To see that the constant~$\sqrt2$ is still best possible in those cases, again use the sequences~$v^{(J)}(c)$ specified in the second part of Lemma~\ref{th:Range}. The~$0$-th component of~$v^{(J)}(c)$ is
\[
 \prod_{j=1}^J \left(1 - |c_j|^2\right)
 = \prod_{j=1}^J \frac{j}{j+1} = \frac{1}{J+1},
\]
which does not vanish,
but tends to~$0$ as~$J \to \infty$.
\end{remark}

\begin{remark}
 When~$K$ is strictly lacunary,
 the argument in Section~\ref{sec:Folding}
 shows that
condition~\eqref{en:Schur}
 in the Schur test
is also satisfied by
 the sequence~$u$ arising from~$v$
itself rather than from~$c$.
This yields inquality~\eqref{eq:PaleyHankelInequality}
with~$C_K = 2$ rather than~$\sqrt2$.
Indeed, 
strict lacunarity makes
the number of indices~$k_i$ in the interval~$[m, 2m)$ 
at most~$1$; moreover,
if there is such an index,
then~$k_{i+1} > 2m$,
and the diagonal term in the~$m$-th row of~$A_v$ vanishes.
The outcome is the 
improvement
\[
\sum_{n=0}^\infty B(m, n)
\le
\max\{1 , \|v\|_2\} + (\|v\|_2)^2
\]
on the estimate~\eqref{eq:MixedUpperBound}.
By the Schur test, the right-hand side above is an upper bound for~$\|A_v\|_\infty$.
Rescaling~$v$ so that~$\|v\|_2 = 1$
makes that upper bound equal to~$2\|v\|_2$.
 
\end{remark}

\bibliographystyle{amsplain}

\end{document}